\documentclass{amsart}

\usepackage{amssymb,amsmath,latexsym,amsthm}
\usepackage[english]{babel}
\usepackage[OT2,T1]{fontenc}
\usepackage{mathtools,amsfonts,amsopn}
\usepackage{mathrsfs}
\usepackage[all,cmtip]{xy}
\usepackage{tikz-cd}
\usetikzlibrary{babel}
\tikzcdset{nodes={execute at begin node=\everymath{\displaystyle}}}
\usepackage[babel]{csquotes}
\usepackage{hyperref}
\hypersetup{colorlinks=true,urlcolor=blue,citecolor=blue,linkcolor=blue}
\usepackage{url}

\newcommand{\Z} {\ensuremath{\mathbf{Z}}}
\newcommand{\Q} {\ensuremath{\mathbf{Q}}}
\newcommand{\R} {\ensuremath{\mathrm{R}}}

\newcommand{\G} {\ensuremath{\mathbf{G}}}

\newcommand{\Ccal} {\ensuremath{\mathscr{C}}}
\newcommand{\Ocal} {\ensuremath{\mathscr{O}}}

\newcommand{\Lcal} {\ensuremath{\mathscr{L}}}

\newcommand{\F} {\ensuremath{\mathbf{F}}}

\renewcommand{\H} {\ensuremath{\mathrm{H}}}
\newcommand{\PPic} {\ensuremath{\mathbf{Pic}}}

\renewcommand{\epsilon}{\varepsilon}
\renewcommand{\theta}{\vartheta}
\renewcommand{\phi}{\varphi}

\DeclareTextFontCommand\textmathbf{\bfseries\boldmath}

\newcommand*\defn[1]{\emph{#1}}

\providecommand{\pfister}[1]{\langle\kern-0.2em\langle#1\rangle\kern-0.2em\rangle}

\providecommand{\iso}{\ensuremath{\cong}}

\providecommand{\isoto}{\xrightarrow{~\sim~}}

\DeclareMathOperator{\im}{im}

\DeclareMathOperator{\Tr}{Tr}

\DeclareMathOperator{\Hom}{Hom}
\DeclareMathOperator{\Mor}{Mor}

\DeclareMathOperator{\colim}{\varinjlim}

\DeclareMathOperator{\Char}{char}
\DeclareMathOperator{\Br}{Br}

\DeclareMathOperator{\Spec}{Spec}

\DeclareMathOperator{\pr}{pr}

\DeclareMathOperator{\NS}{NS}

\DeclareMathOperator{\Pic}{Pic}
\DeclareMathOperator{\Cl}{Cl}

\DeclareMathOperator{\Quot}{Quot}

\DeclareMathOperator{\Divv}{Div}

\providecommand{\Acal}{\ensuremath{\mathscr{A}}}

\providecommand{\Fcal}{\ensuremath{\mathscr{F}}}

\providecommand{\Lcal}{\ensuremath{\mathscr{L}}}
\providecommand{\Mcal}{\ensuremath{\mathscr{M}}}
\providecommand{\Ocal}{\ensuremath{\mathscr{O}}}
\providecommand{\Ucal}{\ensuremath{\mathscr{U}}}

\providecommand{\HHom}{\ensuremath{\mathscr{H}\mathrm{om}}}

\providecommand{\Het}{\ensuremath{\H_\mathrm{\acute{e}t}}}

\newcommand{\tors} {\ensuremath{\mathrm{tors}}}
\newcommand{\Tors} {\ensuremath{\mathrm{n\text{-}tors}}}
\newcommand{\Div} {\ensuremath{\mathrm{n\text{-}div}}}
\providecommand{\divv}{\ensuremath{\mathrm{div}}}

\providecommand{\Zar}{\ensuremath{\mathrm{Zar}}}
\providecommand{\et}{\ensuremath{\mathrm{\acute{e}t}}}
\providecommand{\fppf}{\ensuremath{\mathrm{fppf}}}

\DeclareSymbolFont{cyrletters}{OT2}{wncyr}{m}{n}
\DeclareMathSymbol{\Sha}{\mathalpha}{cyrletters}{"58}

\newtheorem*{maintheorem}{Theorem}

\newtheorem{thm}{Theorem}[section]
\newtheorem{theorem}[thm]{Theorem}
\newtheorem{lemma}[thm]{Lemma}
\newtheorem{corollary}[thm]{Corollary}

\theoremstyle{definition}
\newtheorem{remark}[thm]{Remark}
\newtheorem{definition}[thm]{Definition}

\usepackage{cleveref}

\crefname{theorem}{Theorem}{Theorems}
\crefname{lemma}{Lemma}{Lemmata}
\crefname{corollary}{Corollary}{Corollaries}
\crefname{proposition}{Proposition}{Propositions}
\crefname{definition}{Definition}{Definitions}
\crefname{conjecture}{Conjecture}{Conjectures}
\crefname{example}{Example}{Examples}
\crefname{algorithm}{Algorithm}{Algorithms}
\crefname{remark}{Remark}{Remarks}

\numberwithin{equation}{section}

\begin{document}

\title[On the $p$-torsion of the Tate-Shafarevich group in characteristic $p > 0$]{On the $p$-torsion of the Tate-Shafarevich group\\ of abelian varieties\\ over higher dimensional bases over finite fields}


\author{Timo Keller}
\address{Timo Keller\\
Universität Bayreuth\\
Lehrstuhl Mathematik II (Computeralgebra)\\
Universitätsstraße 30\\
95440 Bayreuth\\
Germany}
\email{Timo.Keller@uni-bayreuth.de}
\urladdr{\url{https://www.timo-keller.de}}

\subjclass[2020]{11G10, 14F20, 14K15}

\keywords{Tate-Shafarevich groups of abelian varieties over higher dimensional bases over finite fields, $p$-torsion in characteristic $p > 0$; Abelian varieties of dimension $> 1$; Étale and other Grothendieck topologies and cohomologies; Arithmetic ground fields for abelian varieties}

\thanks{I thank Jean-Louis Colliot-Thélène, Aise Johan de Jong, Ulrich Görtz, Uwe Jannsen and Moritz Kerz, and Ariyan Javanpeykar, darx and David Loeffler from mathoverflow. I also thank the anonymous referee for improving the exposition. T.\@ K.\@ is supported by the Deutsche Forschungsgemeinschaft (DFG), Projektnummer STO 299/18-1, AOBJ: 667349 while working on this article.}

\maketitle

\begin{abstract}
We prove a finiteness theorem for the first flat cohomology group of finite flat group schemes over integral normal proper varieties over finite fields. As a consequence, we can prove the invariance of the finiteness of the Tate-Shafarevich group of Abelian schemes over higher dimensional bases under isogenies and alterations over/of such bases for the $p$-part. Along the way, we generalize previous results on the Tate-Shafarevich group in this situation.
\end{abstract}

\bigskip
\section{Introduction}
The Tate-Shafarevich group $\Sha(\Acal/X)$ of an Abelian scheme $\Acal$ over a base scheme $X$ is of great importance for the arithmetic of $\Acal$.  It classifies everywhere locally trivial $\Acal$-torsors. Its finiteness is sufficient to establish our analogue of the conjecture of Birch and Swinnerton-Dyer~\cite{KellerGoodReduction} over higher dimensional bases over finite fields.

In~\cite[section~4.3]{KellerSha}, we showed that finiteness of an $\ell$-primary component of the Tate-Shafarevich group descends under generically étale alterations of generical degree prime to $\ell$ for $\ell$ invertible on the base scheme.  This is used in~\cite[Corollary~5.11]{KellerGoodReduction} to prove the finiteness of the Tate-Shafarevich group and an analogue of the Birch-Swinnerton-Dyer conjecture for certain Abelian schemes over higher dimensional bases over finite fields under mild conditions.  In~\cite[section~4.4]{KellerSha}, we showed that finiteness of an $\ell$-primary component of the Tate-Shafarevich group is invariant under étale isogenies. In this article, we prove these results also for the $p^\infty$-torsion.

Recall that we defined the Tate-Shafarevich group of $\Acal/X$ for $\dim{X} > 1$ and $\Acal$ of good reduction as $\Het^1(X,\Acal)$. In~\cite[Lemma~4.15]{KellerSha} we proved as a hypothesis in~\cite[Theorem~4.5]{KellerSha}:
\begin{lemma}
	Let $X/k$ be a smooth variety and $\Ccal/X$ a smooth proper relative curve.  Assume $\dim{X} \leq 2$.  Let $Z \hookrightarrow X$ be a reduced closed subscheme of codimension $\geq 2$.  Then
	\[
	\H^i_Z(X, \PPic^0_{\Ccal/X}) = 0\quad\text{for $i \leq 2$}.
	\]
	If $\dim{X} > 2$, this holds at least up to $p$-torsion.
\end{lemma}
We weaken the hypothesis that $\Acal$ is a Jacobian in~\cite[Lemma~4.15]{KellerSha} to arbitrary Abelian schemes:
\begin{maintheorem} (See~\cref{vanishing-hypothesis})
	Let $X$ be a regular integral Noetherian separated scheme and $\Acal/X$ be an Abelian scheme.  Let $Z \hookrightarrow X$ be a closed subscheme of codimension $\geq 2$.  Then the vanishing condition~\cite[(4.4)]{KellerSha} holds for $\Acal/X$: $\H^i_Z(X,\Acal)$ is torsion for all $i$. Furthermore, $\H^0_Z(X,\Acal) = 0$, and for $i = 1,2$, the only possible torsion is $p$-torsion for $p$ not invertible on $X$.
\end{maintheorem}

Our main results are now as follows:

\begin{maintheorem} (See~\cref{thm:H1fppfofFiniteFlatGroupSchemeOverFiniteFieldIsFinite})
	Let $X$ be a proper integral normal variety over a finite field and $G/X$ be a finite flat commutative group scheme.  Then $\H^1_\fppf(X,G)$ is finite.
\end{maintheorem}
This theorem is proven by reduction to the finite flat simple group schemes $\Z/p$,$\mu_p$ and $\alpha_p$ over an algebraically closed field using de Jong's alteration theorem, Raynaud-Gruson and a dévissage argument.

Using this technical result and refining our methods from~\cite{KellerSha}, we obtain the following three results:

The following theorem has been proved as~\cite[Lemma~4.28]{KellerSha} for $p$ prime to the characteristic of $k$; in this article, we prove it also for $p$ equal to the characteristic of $k$:
\begin{maintheorem} (See~\cref{lemma:ShapinftyCofinitelyGenerated})
	Let $\Acal/X$ be an Abelian scheme over a proper variety $X$ over a finite field of characteristic $p$.  Then $\Sha(\Acal/X)[p^\infty]$ is cofinitely generated.
\end{maintheorem}

In~\cite[Theorem~4.31]{KellerSha}, we proved:
\begin{theorem}
	Let $X/k$ be proper, $\Acal$ and $\Acal'$ Abelian schemes a variety $X$ over a finite field and $f: \Acal' \to \Acal$ an étale isogeny.  Let $\ell \neq \Char{k}$ be a prime.  Then $\Sha(\Acal/X)[\ell^\infty]$ is finite if and only if $\Sha(\Acal'/X)[\ell^\infty]$ is finite.
\end{theorem}
In this article, we prove it also for $\ell$ equal to the characteristic of $k$:
\begin{maintheorem}[invariance of finiteness of $\Sha$ under isogenies, \cref{thm:IsogenyInvarianceOfFinitenessOfSha}]
	Let $X/k$ be a proper variety over a finite field $k$ and $f: \Acal \to \Acal'$ be an isogeny of Abelian schemes over $X$.  Let $p$ be an arbitrary prime.  Assume $f$ étale if $p \neq \Char{k}$.  Then $\Sha(\Acal/X)[p^\infty]$ is finite if and only if $\Sha(\Acal'/X)[p^\infty]$ is finite.
\end{maintheorem}

In~\cite[Theorem~4.29]{KellerSha}, we proved:
\begin{theorem}
	Let $f: X' \to X$ be a morphism of normal integral varieties over a finite field which is an alteration of degree prime to $\ell$ for a prime $\ell$ invertible on $X$, i.e., $f$ is a proper, surjective, generically étale morphism of generical degree prime to $\ell$.  If $\Acal$ is an Abelian scheme on $X$ such that the $\ell^\infty$-torsion of the Tate-Shafarevich group $\Sha(\Acal'/X')$ of $\Acal' := f^*\Acal = \Acal \times_X X'$ is finite, then the $\ell^\infty$-torsion of the Tate-Shafarevich group $\Sha(\Acal/X)$ is finite.
\end{theorem}
In this article, we prove it also for $\ell$ equal to the characteristic of $k$ and remove the condition that the generical degree is prime to $\ell$ if $\ell$ is invertible on $X$:
\begin{maintheorem} (invariance of finiteness of $\Sha$ under alterations, \cref{thm:isotrivialppart} and~\cref{descent-of-finiteness-of-Sha-under-alterations})
	Let $f: X' \to X$ be a proper, surjective, generically finite morphism of generical degree $d$ of regular, integral, separated varieties over a finite field of characteristic $p > 0$.  Let $\Acal$ be an abelian scheme on $X$ and $\Acal' := f^*\Acal = \Acal \times_X X'$.  Let $\ell$ be an arbitrary prime. Assume $(d,\ell) = 1$ if $\ell = p$.  If $\Sha(\Acal'/X')[\ell^\infty]$ is finite, so is $\Sha(\Acal/X)[\ell^\infty]$.
\end{maintheorem}

\paragraph{Notation.}
Canonical isomorphisms are often denoted by ``$=$''. We denote Pontrjagin duality by $(-)^D$ and duals of Abelian schemes and Cartier duals by $(-)^t$.

For a scheme $X$, we denote the set of codimension-$1$ points by $X^{(i)}$ and the set of closed points by $|X|$.

For an abelian group $A$, let $A_\tors$ be the torsion subgroup of $A$, and $A_\Tors = A/A_\tors$.  For $A$ a cofinitely generated $\ell$-primary group, let $A_\divv$ be the maximal divisible subgroup of $A$, which equals the subgroup of divisible elements of $A$ in this case (\cite[Lemma~2.1.1\,(iii)]{KellerGoodReduction}), and $A_\Div = A/A_\divv$.  For an integer $n$ and an object $A$ of an abelian category, denote the cokernel of $A \stackrel{n}{\to} A$ by $A/n$ and its kernel by $A[n]$, and for a prime $p$ the $p$-primary subgroup $\colim_n A[p^n]$ by $A[p^\infty]$.  Write $A[\text{non-}p]$ or $A[p']$ for $\varinjlim_{p \nmid n}A[n]$.  For a prime $\ell$, let the $\ell$-adic Tate module $T_\ell A$ be $\varprojlim_n A[\ell^n]$ and the rationalized $\ell$-adic Tate module $V_\ell A = T_\ell A \otimes_{\Z_\ell} \Q_\ell$.  The corank of $A[p^\infty]$ is the $\Z_p$-rank of $A[p^\infty]^D = T_pA$.

\section{Vanishing of étale cohomology with supports of Abelian schemes}

This is a complement to the ``vanishing condition'' $\H^i_Z(X,G) = 0$ from~\cite[(4.4)]{KellerSha}, which is proven there only for Jacobians of curves, see~\cite[Lemma 4.10]{KellerSha}.

\begin{theorem} \label[theorem]{thm:HiZXG=0}
	Let $X$ be a regular integral Noetherian separated scheme and $G/X$ be a finite étale commutative group scheme of order invertible on $X$.  Let $Z \hookrightarrow X$ be a closed subscheme of codimension $\geq 2$.  Then $\H^i_Z(X,G) = 0$ for $i \leq 2$ (étale cohomology with supports in $Z$).
\end{theorem}
\begin{proof}
	Let $U = X \setminus Z$.  One has a long exact cohomology sequence
	\begin{align*}
	\ldots \to \H^{i-1}(X,G) \to \H^{i-1}(U,G) \to \\\quad\H^i_Z(X,G) \to \H^i(X,G) \to \H^i(U,G) \to \ldots,
	\end{align*}
	so one has to prove that $\H^i(X,G) \to \H^i(U,G)$ is an isomorphism for $i = 0,1$ and injective for $i = 2$.
	
	For $i = 0$, the claim $\H^i_Z(X, G) = 0$ is equivalent to the injectivity of
	\[
	\H^0(X, G) \to \H^0(U, G),
	\]
	which is clear from~\cite[p.~105, Exercise~II.4.2]{HartshorneAG} since $G/X$ is separated, $X$ is reduced and $U \hookrightarrow X$ is dense.
	
	For $i = 1$ the claim $\H^i_Z(X, G) = 0$ is equivalent to
	\[
	\H^0(X, G) \to \H^0(U, G)
	\]
	being surjective and
	\[
	\H^1(X, G) \to \H^1(U, G)
	\]
	being injective. The surjectivity of $\H^0(X, G) \to \H^0(U, G)$ follows e.\,g.\ from
	\begin{theorem} \label[theorem]{thm:rationalmapextends}
		Let $S$ be a normal Noetherian base scheme, and let $u: T \dashrightarrow G$ be an $S$-rational map from a smooth $S$-scheme $T$ to a smooth and separated $S$-group scheme $G$. Then, if $u$ is defined in codimension $\leq 1$, it is defined everywhere.
	\end{theorem}
	\begin{proof}
		See~\cite[p.~109, Theorem~1]{BLR}.
	\end{proof}
	
	For the injectivity of $\H^1(X, G) \to \H^1(U, G)$: If a principal homogeneous space $P/X$ for $G/X$ is trivial over $U$, then it is trivial over $X$:  The trivialization over $U$ gives a rational map from $X$ to the principal homogeneous space and any such map (with $X$ a regular scheme) extends to a morphism by~\cref{thm:rationalmapextends}.
	
	For the surjectivity of $\H^1(X, G) \to \H^1(U, G)$:  This means that any principal homogeneous space $P/U$ extends to a principal homogeneous space $\bar{P}/X$.  By~\cite[p.~123, Corollary~III.4.7]{MilneEtaleCohomology}, we have $\mathrm{PHS}(G/X) \isoto \H^1(X_{\mathrm{fl}}, G)$ (\v{C}ech cohomology) since $G/X$ is affine.  Since $G/X$ is smooth, \cite[p.~123, Remark~III.4.8\,(a)]{MilneEtaleCohomology} shows that we can take étale cohomology as well, and by~\cite[p.~101, Corollary~III.2.10]{MilneEtaleCohomology}, one can take derived functor cohomology instead of \v{C}ech cohomology.  Recall:
	
	\begin{theorem}[Zariski-Nagata purity]\label[theorem]{Zariski-Nagata purity}
		Let $X$ be a locally Noetherian regular scheme and $U$ an open subscheme with closed complement of codimension $\geq 2$. Then the functor $X' \mapsto X' \times_X U$ is an equivalence of categories from \'{e}tale coverings of $X$ to \'{e}tale coverings of $U$.
	\end{theorem}
	\begin{proof}
		See~\cite[Exp.~X, Corollaire~3.3]{SGA1}.
	\end{proof}
	By~\cref{Zariski-Nagata purity}, one can extend $P/U$ uniquely to a $\bar{P}/X$, for which we have to show that it represents an element of $\H^1(X,G)$, i.\,e., that it is a $G$-torsor.
	
	So we need to show that if $P/U$ is an $G|_U$-torsor and $\bar{P}$ an extension of $P$ to a finite étale covering of $X$, then $\bar{P}/X$ is also an $G$-torsor.  For this, we use the following
	\begin{theorem} \label[theorem]{thm:TorsorBedingung}
		Let $X$ be a connected scheme, $G \to X$ a finite flat group scheme, and $\bar{P} \to X$ a scheme over $X$ equipped with a left action $\rho: G \times_X \bar{P} \to \bar{P}$.  These data define a $G$-torsor over $X$ if and only if there exists a finite locally free surjective morphism $Y \to X$ such that $\bar{P} \times_X Y \to Y$ is isomorphic, as a $Y$-scheme with $G \times_X Y$-action, to $G \times_X Y$ acting on itself by left translations.
	\end{theorem}
	\begin{proof}
		See~\cite[p.~171, Lemma~5.3.13]{SzamuelyGalois}.
	\end{proof}
	That $P/U$ is an $G|_U$-torsor amounts to saying that there is an operation
	\[
	G|_U \times_U P \to P
	\]
	as in the previous~\cref{thm:TorsorBedingung}.  Since this is étale locally isomorphic to the canonical action
	\[
	G|_U \times_U G|_U \stackrel{\mu}{\to} G|_U
	\]
	which is finite étale, by faithfully flat descent the operation defines an étale covering, so extends by Zariski-Nagata purity (\cref{Zariski-Nagata purity}) uniquely to an étale covering $H \to X$, which by uniqueness has to be isomorphic to $G \times_{X} \bar{P} \to \bar{P}$.  Now a routine check shows the condition in~\cref{thm:TorsorBedingung}.
	
	There is a finite étale Galois covering $X'/X$ with Galois group $G$ such that $G \times_X X'$ is isomorphic to a direct sum of $\mu_n$ with $n$ invertible on $X$.  The Leray spectral sequence with supports $\H^p(G,\H^q_{Z'}(X',G \times_X X')) \Rightarrow \H^{p+q}_Z(X,G)$ from~\cite[p.~228, Theorem~4.9]{KellerSha}, so it suffices to show the vanishing $\H^q_Z(X',G \times_X X') = 0$ for $q = 0,1,2$.  Hence one can assume $G \iso \mu_n$ for $n$ invertible on $X$.
	
	By~\cite[Example~III.2.22]{MilneEtaleCohomology}, one has an injection $\Br(X) \hookrightarrow \Br(K(X))$ with $K(X)$ the function field of $X$ and $\Br(X) \to \Br(U) \to \Br(K(X))$, so $\Br(X) \to \Br(U)$ is injective.  By the hypotheses on $X$ and since the codimension of $Z$ in $X$ is $\geq 2$, there is a restriction isomorphism $\Pic(X) \isoto \Pic(U)$ (because of the codimension condition and~\cite[Proposition~II.6.5\,(b)]{HartshorneAG}, $\Cl{X} \isoto \Cl{U}$, and because of~\cite[Proposition~II.6.16]{HartshorneAG}, $\Cl{X} \iso \Pic{X}$ functorial in the scheme).  Hence the snake lemma applied to the diagram
	\[\begin{tikzcd}
		0 \ar[r] &
		\Pic(X)/n \ar[d] \ar[r] &
		\H^2(X,\mu_n) \ar[d] \ar[r] &
		\Br(X)[n] \ar[d,hookrightarrow] \ar[r] &
		0  \\
		0 \ar[r] &
		\Pic(U)/n \ar[r] &
		\H^2(U,\mu_n) \ar[r] &
		\Br(U)[n] \ar[r] &
		0
	\end{tikzcd}
	\]
	gives that $\H^2(X,\mu_n) \to \H^2(U,\mu_n)$ is injective, so $\H^2_Z(X,\mu_n) = 0$.
\end{proof}

\begin{corollary}\label[corollary]{vanishing-hypothesis}
	Let $X$ be a regular integral Noetherian separated scheme and $\Acal/X$ be an Abelian scheme.  Let $Z \hookrightarrow X$ be a closed subscheme of codimension $\geq 2$.  Then $\H^i_Z(X,\Acal)$ is torsion for all $i$. Furthermore, $\H^0_Z(X,\Acal) = 0$, and for $i = 1,2$, the only possible torsion is $p$-torsion for $p$ not invertible on $X$.
\end{corollary}
\begin{proof}
	By~\cite[p.~224, Proposition~4.1]{KellerSha}, $\H^i(X,\Acal)$ is torsion for $i > 0$.  The Kummer exact sequence $0 \to \Acal[n] \to \Acal \to \Acal \to 0$ for $n$ invertible on $X$ yields a surjection
	\[
	\H^i_Z(X,\Acal[n]) \twoheadrightarrow \H^i_Z(X,\Acal)[n],
	\]
	so it suffices to show that $\H_Z^i(X,\Acal[n]) = 0$ for $i=1,2$.  But this is~\cref{thm:HiZXG=0}.  The triviality $\H_Z^0(X,\Acal) = 0$ is equivalent to the injectivity of
	\[
	\H^0(X, \Acal) \to \H^0(U, \Acal),
	\]
	which is clear from~\cite[p.~105, Exercise~II.4.2]{HartshorneAG} since $\Acal/X$ is separated, $X$ is reduced and $U \hookrightarrow X$ is dense.
\end{proof}

With vanishing condition~(4.4) in~\cite[Theorem~4.5]{KellerSha} satisfied for $\Acal/X$ by~\cref{vanishing-hypothesis}, the statement there generalizes from $\Acal$ a Jacobian to $\Acal$ a general Abelian scheme:

\begin{theorem} \label[theorem]{thm:codim1Punktereichen}
	Let $X$ be regular, Noetherian, integral and separated and let $\Acal$ be an Abelian scheme over $X$.  For $x \in X$, denote the function field of $X$ by $K$, the quotient field of the strict Henselization of $\Ocal_{X,x}$ by $K_x^{nr}$, the inclusion of the generic point by $j: \{\eta\} \hookrightarrow X$ and let $j_x: \Spec(K_x^{nr}) \hookrightarrow \Spec(\Ocal_{X,x}^{sh}) \hookrightarrow X$ be the composition. Then we have
	\[
	\H^1(X, \Acal) \isoto \ker\Big(\H^1(K, j^*\Acal) \to \prod_{x \in X} \H^1(K_x^{nr}, j_x^*\Acal) \Big).
	\]
	
	One can replace the product over all points by the following:
	
	(a) the closed points $x \in |X|$:  One has isomorphisms
	\begin{align*}
		\H^1(X, \Acal) \isoto \ker\Big(\H^1(K, j^*\Acal) \to \prod_{x \in |X|} \H^1(K_x^{nr}, j_x^*\Acal) \Big)
	\end{align*}
	and
	\begin{align*}
		\H^1(X, \Acal) \isoto \ker\Big(\H^1(K, j^*\Acal) \to \prod_{x \in |X|} \H^1(K_x^{h}, j_x^*\Acal) \Big)
	\end{align*}
	with $K_x^h = \Quot(\Ocal_{X,x}^h)$ the quotient field of the Henselization if $\kappa(x)$ is finite.
	
	or (b) the codimension-$1$ points $x \in X^{(1)}$:  One has an isomorphism
	\begin{align*}
		\H^1(X, \Acal) \isoto \ker\Big(\H^1(K, j^*\Acal) \to \bigoplus_{x \in X^{(1)}} \H^1(K_x^{nr}, j_x^*\Acal) \Big)
	\end{align*}
	if one disregards the $p$-torsion ($p = \Char{k}$) and $X/k$ is smooth projective over $k$ finitely generated.  For $\dim{X} \leq 2$, this also holds for the $p$-torsion.
	
	For $x \in X^{(1)}$, one can also replace $K_x^{nr}$ and $K_x^h$ by the quotient field of the completions $\hat{\Ocal}_{X,x}^{sh}$ and $\hat{\Ocal}_{X,x}^h$, respectively.
\end{theorem}

\section{Finiteness theorems for $\H^1_\fppf$ over finite fields}

The aim of this section is to show that $\H^1_\fppf(X,G)$ is finite for $X$ a normal proper variety over a finite field of characteristic $p$ and $G/X$ a finite flat group scheme.

The proof is by reduction to the case of a finite flat \emph{simple} group scheme over an algebraically closed field, which is isomorphic to $\Z/\ell$ (étale-étale), $\Z/p$ (étale-local), $\mu_p$ (local-étale) or $\alpha_p$ (local-local).

We use the interpretation of $\H^1_\fppf(X,G)$ as $G$-torsors on $X$~\cite[Proposition~III.4.7]{MilneEtaleCohomology} since $G/X$ is affine.  We also exploit de Jong's alteration theorem~\cite[Theorem~4.1]{deJongAlterations}.

Let us first recall some well-known facts on flat cohomology.

\begin{definition} \label[definition]{def:isogeny}
	An \defn{isogeny} of commutative group schemes $G,H$ of finite type over an arbitrary base scheme $X$ is a group scheme homomorphism $f: G \to H$ such that for all $x \in X$, the induced homomorphism $f_x: G_x \to H_x$ on the fibers over $x$ is finite and surjective on identity components.
\end{definition}

\begin{remark}
	See~\cite[p.~180, Definition~4]{BLR}.  We will usually consider isogenies between abelian schemes, for example the finite flat $n$-multiplication, which is étale iff $n$ is invertible on the base scheme or the abelian schemes are trivial.
\end{remark}

\begin{lemma} \label[lemma]{groupschemes-epi}
	Let $G, G'$ be commutative group schemes over a scheme $X$ which are smooth and of finite type over $X$ with connected fibers and $\dim G = \dim G'$ and let $f: G' \to G$ be a morphism of commutative group schemes over $X$.
	
	If $f$ is flat (respectively, étale) then $\ker(f)$ is a flat (respectively, étale) group scheme over $X$, $f$ is quasi-finite, surjective and defines an epimorphism in the category of flat (respectively, étale) sheaves over $X$.
\end{lemma}
\begin{proof}
	See~\cite[Lemma~2.3.3]{KellerGoodReduction}.
\end{proof}

\begin{lemma}[Kummer sequence] \label[lemma]{lemma:pnKummer}
	Let $f: G \to G'$ be a faithfully flat isogeny between smooth commutative group schemes over a base scheme $X$. Then the sequence
	\begin{align*}
		0 \to \ker(f) \to G \stackrel{f}{\to} G' \to 0
	\end{align*}
	is exact on $X_\fppf$.
	This applies in particular to $G = \G_m$ and $G = \Acal$ an abelian scheme and the $n$-multiplication morphism for arbitrary $n \neq 0$.
\end{lemma}
\begin{proof}
	Since $f$ is faithfully flat, in particular surjective, it is an epimorphism of sheaves by~\cref{groupschemes-epi}.  An isogeny of abelian schemes is faithfully flat by~\cite[Proposition~8.1]{MilneAbelianVarieties}.
\end{proof}

\begin{lemma} \label[lemma]{lemma:VergleichHifppfHiSYNHietSmoothGroupScheme}
	Let $G/X$ be a \emph{smooth} commutative group scheme.  Then there are comparison isomorphisms
	\begin{align*}
		\H^i_\fppf(X,G) = \H^i_\et(X,G).
	\end{align*}
	In particular, $\H^i_\fppf(X,G)$ is finite if $X$ is proper over a finite field and $G$ is a commutative finite étale group scheme.
\end{lemma}
\begin{proof}
	See~\cite[Remark~III.3.11\,(b)]{MilneEtaleCohomology} and note that the proof given there gives a comparison isomorphism for any topologies between the étale and the flat site.
\end{proof}

\begin{lemma}
	Let $X$ be a Noetherian integral scheme with function field $K(X)$ and $U \subseteq X$ dense open.  Then there is an exact sequence
	\[
	1 \to \G_m(X) \to \G_m(U) \to \bigoplus_{D \in (X \setminus U)^{(1)}}\Z[D] \to \Cl(X) \to \Cl(U) \to 0.
	\]
\end{lemma}
\begin{proof}
	The assumptions imply that there is a commutative diagram with exact rows
	\[\begin{tikzcd}
		1 \arrow[r] &\Ocal_X(X)^\times \arrow[d]\arrow[r] & K(X)^\times \arrow[d, equal]\arrow[r] & \Divv(X) \arrow[d, twoheadrightarrow]\arrow[r] & \Cl(X) \arrow[d]\arrow[r] & 0\\
		1 \arrow[r] &\Ocal_X(U)^\times \arrow[r] & K(U)^\times \arrow[r] & \Divv(U) \arrow[r] & \Cl(U) \arrow[r] & 0.
	\end{tikzcd}\]
	A diagram chase yields the result.
\end{proof}

\begin{corollary} \label[corollary]{cor:GmUndPicForOpenSubscheme}
	Let $X$ be a Noetherian integral regular scheme and let $U \subseteq X$ be dense open.  Then there is an exact sequence
	\[
	1 \to \G_m(X) \to \G_m(U) \to \bigoplus_{D \in (X \setminus U)^{(1)}}\Z[D] \to \Pic(X) \to \Pic(U) \to 0.
	\]
\end{corollary}
\begin{proof}
	By the assumptions, $\Cl(X) = \Pic(X)$ and $\Cl(U) = \Pic(U)$.
\end{proof}

\begin{corollary} \label[corollary]{cor:H1fppfmupFinite}
	Let $X/\F_q$ be an integral Noetherian regular proper variety and let $j: U \hookrightarrow X$ be the inclusion of an open subscheme of $X$.  Then $\H^1_\fppf(U,\mu_{p^n})$ is finite for all $n$ and any prime $p$.
\end{corollary}
\begin{proof}
	The Kummer sequence~\cref{lemma:pnKummer} on $U_\fppf$ together with \[\Pic(U) = \H^1_\fppf(U,\G_{m,U})\] by~\cref{lemma:VergleichHifppfHiSYNHietSmoothGroupScheme} yields the exact sequence
	\[
	1 \to \G_m(U)/p^n \to \H^1_\fppf(U,\mu_{p^n}) \to \Pic(U)[p^n] \to 0.
	\]
	Since $\G_m(X) = \Gamma(X,\G_m)^\times$ is finite by the coherence theorem~\cite[Thm. (3.2.1)]{EGAIII1}, since $X/\F_q$ is proper and $\F_q$ is finite, and since $\Pic(X)$ is finitely generated since its sits in a short exact sequence $0 \to \Pic^0(X) \to \Pic(X) \to \NS(X) \to 0$ and $\Pic^0(X)$ is finite since it is the group of rational points of an Abelian variety over a finite field and $\NS(X)$ is always finitely generated by~\cite[Exp.~XIII, §\,5]{SGA6}, by~\cref{cor:GmUndPicForOpenSubscheme} and the finiteness of $(X \setminus U)^{(1)}$, this exact sequence gives the finiteness of $\G_m(U)/p^n$ and of $\Pic(U)[p^n]$.
\end{proof}

The following statements and proofs in this section are an extended version of the sketch of~\cref{thm:H1fppfofFiniteFlatGroupSchemeOverFiniteFieldIsFinite} given by \enquote*{darx} in~\cite{257441}.

\begin{lemma} \label[lemma]{lemma:TorsorOverGenericPointTrivial}
	Let $X$ be a normal integral scheme and $G/X$ be a finite flat group scheme.  If $T$ is a $G$-torsor on $X$ trivial over the generic point of $X$, then $T$ is trivial.  Hence, $\H^1_\fppf(X,G) \to \H^1_\fppf(K(X),G)$ is injective, and if $f: Y \to X$ is birational, $f^*: \H^1_\fppf(X,G) \to \H^1_\fppf(Y,G)$ is injective.
\end{lemma}
\begin{proof}
	Since $T$ is trivial over the generic point of $X$, generically, there is a section of $\pi: T \to X$.  This extends to a rational map $\sigma: X \dashrightarrow T$.  Take the schematic closure $i: X' \hookrightarrow T$ of $\sigma$.  The composition $\pi \circ i: X' \to T \to X$ is birational and finite (as a composition of a closed immersion and a finite morphism).  By~\cite[Corollary~12.88]{Goertz-Wedhorn}, since $X$ is normal, $X' \to X$ is an isomorphism.  Hence $\sigma$ is a section of $\pi$, so $T/X$ is trivial.
\end{proof}

\begin{lemma} \label[lemma]{lemma:FiniteNumberOfSectionsOfFiniteFlatScheme}
	Let $X$ be a proper variety over a finite field and $Y/X$ be a finite flat scheme.  Let $Z/X$ be proper.  Then $Y(Z)$ is finite.
\end{lemma}
\begin{proof}
	Since $\Mor_X(Z,Y) = \Mor_Z(Z, Y \times_X Z)$, one can assume $Z = X$.  So we have to show that there are only finitely many sections to $\pi: Y \to X$.  Such a section corresponds to an $\Ocal_X$-algebra map $\pi_*\Ocal_Y \to \Ocal_X$.  But $\H^0_\Zar(X, \HHom_X(\pi_*\Ocal_Y, \Ocal_X))$ is finite by the coherence theorem~\cite[Thm.~(3.2.1)]{EGAIII1} as it is a finite dimensional vector space over a finite field.
\end{proof}

\begin{lemma} \label[lemma]{lemma:KernelFinite}
	Let $Y \to X$ be an alteration of proper integral varieties with $X$ normal, and $G/X$ be a finite flat commutative group scheme.  Then $\ker(\H^1_\fppf(X,G) \to \H^1_\fppf(Y,G))$ is finite.  Hence $\H^1_\fppf(X,G)$ is finite if $\H^1_\fppf(Y,G)$ is.
\end{lemma}
\begin{proof}
	If $Y \to X$ is a blow-up, the kernel is trivial by~\cref{lemma:TorsorOverGenericPointTrivial} since a blow-up is birational.  Hence the statement holds for blow-ups.
	
	By~\cite[Théorème~5.2.2]{RaynaudGruson}, there is a blow-up $f: X' \to X$ such that $Y' := Y \times_X X'$ is flat over $X'$.  Since a normalization morphism of integral schemes is birational~\cite[Proposition~4.1.22]{Liu2006}, one can assume $X'$ normal.  There is a commutative diagram
	\[\begin{tikzcd}
		0 \ar[r] & \ker(\H^1_\fppf(X,G) \to \H^1_\fppf(Y,G)) \ar[r] \ar[d] & \H^1_\fppf(X,G) \ar[d,"f^*"] \\
		0 \ar[r] & \ker(\H^1_\fppf(X',G) \to \H^1_\fppf(Y',G)) \ar[r] & \H^1_\fppf(X',G)
	\end{tikzcd}
	\]
	By the snake lemma, since $\ker{f^*}$ is finite as $f$ is a blow-up, \[\ker(\H^1_\fppf(X,G) \to \H^1_\fppf(Y,G))\] is finite if we can show that \[\ker(\H^1_\fppf(X',G) \to \H^1_\fppf(Y',G))\] is finite.  Hence, we can assume $Y \to X$ finite flat.
	
	Let $T \to X$ be in the kernel, i.\,e., it is a $G$-torsor on $X$ trivial when pulled back to $Y$.  Choose a section $\sigma: Y \to T \times_X Y$; there are only finitely many of them by~\cref{lemma:FiniteNumberOfSectionsOfFiniteFlatScheme}. Two such sections differ by an element of $G(Y)$.  Since the base change $T \times_X (Y \times_X Y) \to Y \times_X Y$ is a $G$-torsor, one can take the $1$-cocycle
	\[
	\tau := d^0(\sigma) = \pr_0^*(\sigma) - \pr_1^*(\sigma) \in G(Y \times_X Y).
	\]
	The section $\tau$ corresponds to the isomorphism class of the $G$-torsor $T$ by the descent theory
	for the fppf covering $\{Y \to X\}$: As $\H^1_\fppf(-,G)$ can be computed by \v{C}ech cohomology and as the class of $T$ in $\H^1_\fppf(X,G) = \check{H}^1_\fppf(X,G) = \colim_\Ucal\check{H}^1_\fppf(\Ucal,G)$ (the colimit taken over the coverings of $X$; the natural morphism from the first \v{C}ech cohomology to the first derived functor cohomology is always an isomorphism) is trivialized by the covering $\{Y \to X\}$, it can be represented as the $1$-cocycle $\tau = d^0(\sigma)$, which is a $1$-coboundary:
	\[
	\check{H}^1(\{Y \to X\},G) = \frac{\ker\big(G(Y \times_X Y) \stackrel{d^1}{\to} G(Y \times_X Y \times_X Y)\big)}{\im\big(G(Y) \stackrel{d^0}{\hookrightarrow} G(Y \times_X Y)\big)}
	\]
	But by~\cref{lemma:FiniteNumberOfSectionsOfFiniteFlatScheme}, $G(Y \times_X Y)$ is finite.
\end{proof}

\begin{lemma} \label[lemma]{lemma:SchemeTheoreticClosureOfFiniteFlatGroupScheme}
	Let $X$ be an integral scheme with function field $K$ and $G/X$ be a finite flat group scheme.  Let $H_K \hookrightarrow G_K$ be a finite flat group scheme.  Then there is a blow-up $\tilde{X}/X$ such that $H_K$ extends to a finite flat subgroup scheme of $G \times_X \tilde{X}$.
\end{lemma}
\begin{proof}
	Let $H \hookrightarrow G$ be the schematic closure of $H_K \hookrightarrow G$.  The morphism $H \to G \to X$ is finite as a composition of a closed immersion and a finite morphism. By~\cite[Théorème~5.2.2]{RaynaudGruson}, there is a blow-up $X' \to X$ such that $H' := H \times_X X' \to X'$ is flat.  Then, $H'$ is the schematic closure of $H_K \hookrightarrow G' := G \times_X X'$.  So one can assume $H/X$ finite flat.
	
	Let $Y \to X$ be finite flat.  Since the morphism is affine, locally, one has the diagram
	\[\begin{tikzcd}
		A \ar[r,hookrightarrow]  &
		A \otimes_R \Quot(R)  \\
		R\ar[r,hookrightarrow] \ar{u} &
		\Quot(R).
		\ar{u}
	\end{tikzcd}
	\]
	Here, the upper horizontal arrow is injective by flatness of $R \to A$.  Hence $Y$ is the schematic closure of $Y_K$ in $Y$.
	
	By flatness, the schematic closure of $H_K \times_K H_K$ in $G \times_X G$ is $H \times_X H$.  By the universal property of the schematic closure~\cite[(10.8)]{Goertz-Wedhorn}, one has the factorization
	\[\begin{tikzcd}
		H_K \times_K H_K\ar[d,hookrightarrow] \ar[r,"\mu"] &
		H_K \ar[d,hookrightarrow]  \\
		H \times_X H\ar[d,hookrightarrow] \ar[dashed,"\mu"]{r} &
		H \ar[d,hookrightarrow]  \\
		G \times_X G \ar[r,"\mu"] &
		G,
	\end{tikzcd}
	\]
	for the multiplication $\mu$, and similar for the inverse and unit section.
\end{proof}

\begin{lemma} \label[lemma]{lemma:FiltrationBypOrderGroupSchemes}
	Let $X$ be a proper integral variety over a field and $G/X$ be a finite flat commutative group scheme.  After an alteration $X' \to X$, there exists a filtration of $G$ by finite flat group schemes with subquotients of prime order.
\end{lemma}
\begin{proof}
	Over the algebraic closure of the function field of $X$, there is such an filtration since the only simple objects in the category of finite flat group schemes of $p$-power order are $\mu_p$, $\Z/p$ and $\alpha_p$.  Since everything is of finite presentation, these are defined over a finite extension of the function field~\cite[Corollary~10.79]{Goertz-Wedhorn}.  Now take the normalization in this finite extension of function fields and use~\cref{lemma:SchemeTheoreticClosureOfFiniteFlatGroupScheme}.
\end{proof}

\begin{theorem} \label[theorem]{thm:H1fppfofFiniteFlatGroupSchemeOverFiniteFieldIsFinite}
	Let $X$ be a proper integral normal variety over a finite field and $G/X$ be a finite flat commutative group scheme.  Then $\H^1_\fppf(X,G)$ is finite.
\end{theorem}
\begin{proof}
	By~\cref{lemma:FiltrationBypOrderGroupSchemes}, \cref{lemma:KernelFinite} and the long exact cohomology sequence one can assume $G$ of prime order $p$ (since the case of $G/X$ étale is easily dealt with).  Since then $G$ is simple by~\cite[ p.~38]{ShatzGroupSchemes} and since $F \circ V = [p] = 0$ by~\cite[p.~62]{ShatzGroupSchemes} and~\cite[p.~141]{MumfordAbelianVarieties}, either $V = 0$ or $F = 0$ on $G$.
	
	If $V = 0$, by~\cite[Proposition~2.2]{deJongDieudonne}, there is a short exact sequence
	\[
	0 \to G \to \Lcal \to \Mcal \to 0
	\]
	with vector bundles $\Lcal, \Mcal$.  By the coherence theorem~\cite[Thm.~(3.2.1)]{EGAIII1}, as $X$ is proper and lives over a finite ground field, and by comparison of Zariski and fppf cohomology~\cite[Proposition~III.3.7]{MilneEtaleCohomology}, the long exact cohomology sequence shows that $\H^i_\fppf(X,G)$ is finite.
	
	If $F = 0$, after replacing $X$ by an alteration by~\cref{lemma:KernelFinite} as in the proof of~\cref{lemma:FiltrationBypOrderGroupSchemes}, one can assume that $G$ is isomorphic to $\mu_p$ over the generic point.  Since for $Y,Z/X$ of finite presentation such that $Y_K \iso Z_K$, there is a non-empty open subscheme $U \hookrightarrow X$ such that $Y_U \iso Z_U$, there is a non-empty open subscheme $U \hookrightarrow X$ such that $G_U \iso \mu_{p,U}$.  By~\cite{deJongAlterations}, there is an alteration $f: X' \to X$ such that $X'$ is regular.  By~\cref{cor:H1fppfmupFinite}, $\H^1_\fppf(f^{-1}(U),\mu_p)$ is finite.  By~\cref{lemma:TorsorOverGenericPointTrivial}, $\H^1_\fppf(X', G \times_X X')$ is finite, so by~\cref{lemma:KernelFinite}, $\H^1_\fppf(X,G)$ is finite.
\end{proof}

\section{Isogeny invariance of finiteness of $\Sha$, the $p$-part}
In this section, we extend~\cite[p.~240, Theorem~4.31]{KellerSha} to $p^\infty$-torsion.

\begin{theorem} \label[theorem]{thm:IsogenyInvarianceOfFinitenessOfSha}
	Let $X/k$ be a proper variety over a finite field $k$ and $f: \Acal \to \Acal'$ be an isogeny of Abelian schemes over $X$.  Let $p$ be an arbitrary prime.  Assume $f$ étale if $p \neq \Char{k}$.  Then $\Sha(\Acal/X)[p^\infty]$ is finite if and only if $\Sha(\Acal'/X)[p^\infty]$ is finite.
\end{theorem}
\begin{proof}
	In the case where $\ell$ is invertible on $X$ and $f$ is étale (i.\,e., of degree invertible on $X$), this is~\cite[p.~240, Theorem~4.31]{KellerSha}.
	
	Now assume $p = \Char{k}$.  Then the short exact sequence of flat sheaves \cref{lemma:pnKummer} yields an exact sequence in cohomology
	\[
	\H^1_\fppf(X,\ker(f)) \to \H^1_\fppf(X,\Acal) \stackrel{f}{\to} \H^1_\fppf(X,\Acal')
	\]
	and note that $\H^1_\fppf(X,\Acal) = \H^1_\et(X,\Acal) = \Sha(\Acal/X)$ by~\cref{lemma:VergleichHifppfHiSYNHietSmoothGroupScheme} since $\Acal/X$ is smooth, and that $\H^1_\fppf(X,\ker(f))$ is finite by~\cref{thm:H1fppfofFiniteFlatGroupSchemeOverFiniteFieldIsFinite}.  Note that all groups are torsion (the Tate-Shafarevich groups by~\cite[p.~224, Proposition~4.1]{KellerSha}), hence the sequence stays exact after taking $p^\infty$-torsion. So $\Sha(\Acal/X)[p^\infty]$ is finite if $\Sha(\Acal'/X)[p^\infty]$ is.
	
	For the converse, note that by~\cite[Proposition~2.19]{KellerGoodReduction}, there is a polarization $\lambda: \Acal^t \to \Acal$.   Hence, the argument above for $\lambda$ and $\lambda^t$ implies that $\Sha(\Acal^t/X)[p^\infty]$ is finite iff $\Sha(\Acal/X)[p^\infty]$ is, and analogously for $\Sha(\Acal'/X)[p^\infty]$.  Taking the dual Kummer sequence $0 \to \ker(f^t) \to \Acal'^t \to \Acal^t \to 0$ yields an exact sequence
	\[
	\H^1_\fppf(X,\ker(f^t)) \to \Sha(\Acal'^t/X) \to \Sha(\Acal^t/X).
	\]
	By the same argument as above, $\Sha(\Acal'^t/X)[p^\infty]$ is finite if $\Sha(\Acal^t/X)[p^\infty]$ is if $\Sha(\Acal/X)[p^\infty]$ is.  So $\Sha(\Acal'/X)[p^\infty]$ is finite.
\end{proof}

\section{Descent of finiteness of $\Sha$, the $p$-part}\label{sec:descent-of-finiteness}
In this section, we extend~\cite[p.~238, Theorem~4.29]{KellerSha} to $p^\infty$-torsion.

\begin{lemma} \label[lemma]{lemma:ShapinftyCofinitelyGenerated}
	Let $\Acal/X$ be an Abelian scheme over a proper variety $X$ over a finite field of characteristic $p$.  Then $\Sha(\Acal/X)[p^\infty]$ is cofinitely generated.
\end{lemma}
Recall that $\Sha(\Acal/X)$ was defined as $\H^1_\et(X,\Acal)$ in~\cite[p.~225, Definition~4.2]{KellerSha}.
\begin{proof}
	The long exact cohomology sequence associated to the Kummer sequence \cref{lemma:pnKummer} gives us a surjection
	\[
	\H^1_\fppf(X,\Acal[p^n]) \twoheadrightarrow \H^1_\fppf(X,\Acal)[p^n] \to 0
	\]
	Now, since $\Acal/X$ is a smooth group scheme, \cref{lemma:VergleichHifppfHiSYNHietSmoothGroupScheme} gives us an isomorphism $\H^1_\fppf(X,\Acal) = \H^1_\et(X,\Acal)$, which by definition equals $\Sha(\Acal/X)$.  By~\cref{thm:H1fppfofFiniteFlatGroupSchemeOverFiniteFieldIsFinite}, $\H^1_\fppf(X,\Acal[p^n])$ is finite since $X/\F_q$ is proper.  From this, one sees that $\H^1_\et(X,\Acal)[p]$ is finite.  Hence $\Sha(\Acal/X)[p^\infty]$ is cofinitely generated by~\cite[Lemma~2.38]{KellerGoodReduction}.
\end{proof}

\begin{lemma}[existence of trace morphism] \label[lemma]{lemma:SpurabbildungInFlacherKohomologie}
	Let $f: X' \to X$ be a finite étale morphism of constant degree $d$ and let $\Fcal$ be an fppf sheaf on $X$.  Then there is a trace map $\Tr_f: f_*f^*\Fcal \to \Fcal$, functorial in $\Fcal$, such that $\phi \mapsto \Tr_f \circ f_*(\phi)$ is an isomorphism $\Hom_{X'}(\Fcal',f^*\Fcal) \to \Hom_X(\pi_*\Fcal',\Fcal)$ for any fppf sheaf $\Fcal'$ on $X'$.  Thus, $f_* = f_!$, that is, $f_*$ is left adjoint to $f^*$, and $\Tr_f$ is the adjunction map.  The composites
	\[
	\Fcal \to f_*f^*\Fcal \stackrel{\Tr_f}{\to} \Fcal
	\]
	and 
	\[ \H^r_\fppf(X,\Fcal) \xrightarrow{f^*} \H^r_\fppf(X',f^*\Fcal) \xrightarrow{\operatorname{can}} \H^r_\fppf(X,f_*f^*\Fcal) \xrightarrow{\Tr_f} \H^r_\fppf(X,\Fcal)
	\]
	are multiplication by $d$.
\end{lemma}
\begin{proof}
	On may copy the proof of~\cite[p.~168, Lemma~V.1.12]{MilneEtaleCohomology} almost verbatim: Let $\Fcal$ be a fppf sheaf on $X$. Let $X'' \to X$ be finite Galois with Galois group $G$ factoring as $X'' \to X' \to X$; $X'' \to X'$ is Galois with Galois group $H \leq G$. For any $U/X$ flat, we have $\Gamma(U,\Fcal) \hookrightarrow \Gamma(U',\Fcal) \hookrightarrow \Gamma(U'',\Fcal)$ and $\Gamma(U,\Fcal) \isoto \Gamma(U'',\Fcal)^G$, where $U' = U \times_X X'$ and $U'' = U \times_X X''$. For a section $s \in \Gamma(U,f_*f^*\Fcal) := \Gamma(U',\Fcal)$, we define
	\[
	\Tr_f(s) := \sum_{\sigma \in G/H}\sigma(s|_{U''});
	\]
	as this is fixed by $G$, it may be regarded as an element of $\Gamma(U,\Fcal) \isoto \Gamma(U'',\Fcal)^G$. Clearly, $\Tr_f$ defines a morphism $f_*f^*\Fcal \to \Fcal$ such that its composite with $\Fcal \to f_*f^*\Fcal$ is multiplication by the degree $d$ of $f$.
	
	If $X'$ is a disjoint union of $d$ copies of $X$, obviously \[\Hom_{X'}(\Fcal',f^*\Fcal) \isoto \Hom_X(f_*\Fcal',\Fcal),\] and one may reduce the question to this split case by passing to a finite \'{e}tale covering of $X$, for example to $X'' \to X$, and using the fact that $\Hom$ is a sheaf.
	
	In
	\[
	\H^r_\fppf(X,\Fcal) \xrightarrow{f^*} \H^r_\fppf(X',f^*\Fcal) \xrightarrow{\operatorname{can}} \H^r_\fppf(X,f_*f^*\Fcal) \xrightarrow{\Tr_f} \H^r_\fppf(X,\Fcal)
	\]
	the composite of the first two maps is induced by $\Fcal \to f_*f^*\Fcal$, and the composite of all three is induced by $(\Fcal \to f_*f^*\Fcal \xrightarrow{\Tr_f} \Fcal)$, which is multiplication by $d$.
\end{proof}

\begin{theorem} \label[theorem]{thm:isotrivialppart}
	Let $p$ be a prime and $X$ be a scheme of characteristic $p$. Let $f: X' \to X$ be a proper, surjective, generically étale morphism of generical degree prime to $p$ of regular, integral, separated varieties over a finite field.  Let $\Acal$ be an abelian scheme on $X$ and $\Acal' := f^*\Acal = \Acal \times_X X'$.  If $\Sha(\Acal'/X')[p^\infty]$ is finite, so is $\Sha(\Acal/X)[p^\infty]$.
\end{theorem}
\begin{proof}
	The same proof as in~\cite[Theorem~4.29]{KellerSha} works, one only needs $\Sha(\Acal/X)[p^\infty]$ to be cofinitely generated in Step~2, which is~\cref{lemma:ShapinftyCofinitelyGenerated}.  The trace morphism in Step~3 for fppf cohomology comes from~\cref{lemma:SpurabbildungInFlacherKohomologie}.  Note that the proof given there does not need the regularity of $X, X'$ and that varieties over a field are excellent by~\cite[Corollary~2.40\,(a)]{Liu2006}. For the convenience of the reader, we reproduce the proof of~\cite[Theorem~4.29]{KellerSha} adapted to our situation here:
	
	\textbf{Step~1:  {\boldmath$\H^1_\fppf(X, f_*\Acal')[p^\infty]$} is finite.}  This follows from the low terms exact sequence
	\[
	0 \to \H^1_\fppf(X, f_*\Acal') \to \H^1_\fppf(X', \Acal')
	\]
	associated to the Leray spectral sequence \[\H^p_\fppf(X, \R^qf_*\Acal') \Rightarrow \H^{p+q}_\fppf(X', \Acal')\] and the finiteness of
	\[
	\H^1_\fppf(X', \Acal')[p^\infty] = \Sha(\Acal'/X')[p^\infty].
	\]
	
	\textbf{Step~2:  The theorem holds if there is a trace morphism.}  Since by \cref{lemma:SpurabbildungInFlacherKohomologie} there is a trace morphism $f_*f^*\Acal \to \Acal$ such that the composition with the adjunction morphism
	\[
	\Acal \to f_*f^*\Acal \to \Acal
	\]
	is multiplication by $\deg{f} \neq 0$, the finiteness of $\H^1_\fppf(X,\Acal)[p^\infty]$ follows from that of $\H^1_\fppf(X,f_*\Acal')[p^\infty]$ because both groups are cofinitely generated by~\cref{lemma:ShapinftyCofinitelyGenerated}.
	
	\textbf{Step~3: Proof of the theorem in the general case.} Let $\eta$ be the generic point of $X$.  Define $X'_\eta$ by the commutativity of the cartesian diagram
	\begin{equation}\begin{tikzcd}
			X'_\eta \ar[r,hookrightarrow,"g'"] \ar[d,"f_\eta"] & X' \ar[d,"f"]\\
			\{\eta\}  \ar[r,hookrightarrow,"g"]           & X.
	\end{tikzcd}\end{equation}
	Since $f$ is generically \'{e}tale, we can apply~\cref{lemma:SpurabbildungInFlacherKohomologie} to $f_\eta$ in this commutative diagram. From the commutativity of that diagram, the kernel of $f^*: \H^1_\fppf(X,\Acal) \to \H^1_\fppf(X',\Acal')$ is contained in the kernel of the composition
	\[
	\H^1_\fppf(X,\Acal) \stackrel{g^*}{\to} \H^1_\fppf(\{\eta\},\Acal_\eta) \stackrel{f_\eta^*}{\to} \H^1_\fppf(X'_\eta,\Acal'_{X'_\eta}),
	\]
	so it suffices to show that the first arrow $g^*$ is injective. But by the N\'{e}ron mapping property $\Acal \isoto g_*g^*\Acal$~\cite[Theorem~3.3]{KellerSha} (for the \emph{\'{e}tale} topology!), $\H^1_\et(X,\Acal) \isoto \H^1_\et(X,g_*\Acal_\eta)$. However, the Leray spectral sequence \[\H^p_\et(X,\R^qg_*\Acal_\eta) \Rightarrow \H^{p+q}_\et(\{\eta\},\Acal_\eta)\] gives an injection
	\[
	0 \to \H^1_\et(X,g_*\Acal_\eta) \to \H^1_\et(\{\eta\},\Acal_\eta).
	\]
	But because $\Acal/X$ and $\Acal_\eta/\{\eta\}$ are smooth commutative group schemes, their \'{e}tale cohomology agrees with their flat cohomology, see~\cref{lemma:VergleichHifppfHiSYNHietSmoothGroupScheme}, and the comparison of topology morphisms are functorial.
\end{proof}

\begin{theorem}[Stein factorization, alteration $=$ finite $\circ$ modification]\label[theorem]{Stein factorization}
	Let $f: X' \to X$ be a proper morphism of Noetherian schemes. Then one can factor $f$ into $g \circ f'$, where $f': X' \to Y := \mathbf{Spec}_Xf_*\Ocal_{X'}$ is a proper morphism with connected fibers, and $g: Y \to X$ is a finite morphism. If $f$ is an alteration, $f'$ is birational and proper (a \emph{modification}).
\end{theorem}
\begin{proof}
	See~\cite[Thm.~4.3.1]{EGAIII1} for the statement on the existence of the factorization, which includes $Y = \mathbf{Spec}_Xf_*\Ocal_{X'}$.
	
	Assume now that $f$ is an alteration. If $U \subseteq X$ is an open subscheme such that $f|_U$ is finite (in particular affine), one may shrink $U$ such that it is affine, so by finiteness of $g$, $g: g^{-1}(U) \to U$ is finite and can be written as $\Spec{B} \to \Spec{A}$. From the statement of Stein factorization, $g^{-1}(U) = \mathbf{Spec}_Uf_*\Ocal_{f^{-1}(U)}$, but $f'$ has geometrically connected fibers, so $f'_*\Ocal_{X'} = \Ocal_Y$, so $f'|_{g^{-1}(U)}$ is an isomorphism because it is affine.
\end{proof}

We also remove the hypotheses that $f$ is generically étale and has degree prime to $\ell$ if $\ell$ is invertible on the base scheme in~\cite[Theorem~4.29]{KellerSha}:

\begin{theorem} \label[theorem]{descent-of-finiteness-of-Sha-under-alterations}
	Let $f: X' \to X$ be a proper, surjective, generically finite morphism of regular, integral, separated varieties over a finite field.  Let $\Acal$ be an abelian scheme on $X$ and $\Acal' := f^*\Acal = \Acal \times_X X'$.  Let $\ell$ be invertible on $X$.  If $\Sha(\Acal'/X')[\ell^\infty]$ is finite, so is $\Sha(\Acal/X)[\ell^\infty]$.
\end{theorem}
\begin{proof}
	By the Stein factorization~\cref{Stein factorization}, $f$ factors as a proper, surjective, birational morphism followed by a finite morphism. In particular, it is a generically étale alteration. The finite morphism factors as a finite purely inseparable morphism followed by a finite generically étale morphism. We prove the finiteness assertion of the theorem for all such morphisms separately:
	
	If $f$ is generically étale, this is~\cref{thm:isotrivialppart}. If $f$ is a proper, surjective, birational morphism, it is generically an isomorphism, i.e., generically étale of degree $1$.
	
	If $f$ is a universal homeomorphism, the \'{e}tale sites of $X$ and $X'$ are equivalent by $f^*$ and $f_*$ by~\cite[VIII.1.1]{SGA42}. In particular, the \'{e}tale cohomology groups $\Sha(\Acal/X) = \Het^1(X,\Acal)$ and $\Sha(\Acal'/X') = \Het^1(X',f^*\Acal)$ are isomorphic via $f^*$.
\end{proof}

\bigskip 

\bibliographystyle{amsalpha}
\bibliography{ShaP}

\providecommand{\bysame}{\leavevmode\hbox to3em{\hrulefill}\thinspace}
\providecommand{\MR}{\relax\ifhmode\unskip\space\fi MR }
\providecommand{\MRhref}[2]{%
  \href{http://www.ams.org/mathscinet-getitem?mr=#1}{#2}
}
\providecommand{\href}[2]{#2}
\begin{thebibliography}{{Sha}86}

\bibitem[AGV72]{SGA42}
Michael Artin, Alexandre Grothendieck, and Jean-Louis Verdier,
  \emph{{S\'{e}minaire de G\'{e}om\'{e}trie Alg\'{e}brique du Bois Marie --
  1963--64 -- Th\'{e}orie des topos et cohomologie \'{e}tale des sch\'{e}mas --
  (SGA 4) -- vol.\ 2}}, Lecture notes in mathematics (in French) \textbf{270},
  Berlin; New York: Springer-Verlag. iv+418, 1972 (french).

\bibitem[BGI71]{SGA6}
Pierre Berthelot, Alexandre Grothendieck, and Luc Illusie, \emph{{S\'{e}minaire
  de G\'{e}om\'{e}trie Alg\'{e}brique du Bois Marie -- 1966--67 -- Th\'{e}orie
  des intersections et th\'{e}orème de Riemann-Roch -- (SGA 6)}}, Lecture
  notes in mathematics (in French) \textbf{225}, Berlin; New York:
  Springer-Verlag. xii+700, 1971 (french).

\bibitem[BLR90]{BLR}
Siegfried Bosch, Werner Lütkebohmert, and Michel Raynaud, \emph{{N\'{e}ron
  models.}}, {Ergebnisse der Mathematik und ihrer Grenzgebiete, 3. Folge,
  \textbf{21}. Berlin etc.: Springer-Verlag. x, 325 p.}, 1990 (english).

\bibitem[{de }93]{deJongDieudonne}
Aise~J. {de Jong}, \emph{{Finite locally free group schemes in characteristic
  $p$ and Dieudonn\'e modules.}}, {Invent. Math.} \textbf{114} (1993), no.~1,
  89--137 (english).

\bibitem[{de }96]{deJongAlterations}
\bysame, \emph{{Smoothness, semi-stability and alterations.}}, {Publ. Math.,
  Inst. Hautes \'Etud. Sci.} \textbf{83} (1996), 51--93 (english).

\bibitem[GD61]{EGAIII1}
Alexandre Grothendieck and Jean Dieudonn\'{e}, \emph{{El\'{e}ments de
  g\'{e}om\'{e}trie alg\'{e}brique\,: III. \'{E}tude cohomologique des
  faisceaux coh\'{e}rents, Première partie.}}, {Publ.\ Math.\ IH\'{E}S}
  \textbf{11} (1961), 5--167 (french).

\bibitem[Gro71]{SGA1}
Alexandre Grothendieck, \emph{{S\'{e}minaire de G\'{e}om\'{e}trie
  Alg\'{e}brique du Bois Marie -- 1960--61 -- Revêtements \'{e}tales et groupe
  fondamental -- (SGA 1)}}, (Lecture notes in mathematics \textbf{224}),
  Berlin; New York: Springer-Verlag. xxii+447, 1971 (french).

\bibitem[GW10]{Goertz-Wedhorn}
Ulrich G{\"o}rtz and Torsten Wedhorn, \emph{{Algebraic geometry I. Schemes.
  With examples and exercises.}}, {Advanced Lectures in Mathematics. Wiesbaden:
  Vieweg+Teubner. vii, 615~p.}, 2010 (english).

\bibitem[Har83]{HartshorneAG}
Robin Hartshorne, \emph{{Algebraic geometry. Corr.\ 3rd printing.}}, {Graduate
  Texts in Mathematics, \textbf{52}. New York-Heidelberg-Berlin:
  Springer-Verlag. XVI, 496 p.}, 1983 (english).

\bibitem[Kel16]{KellerSha}
Timo Keller, \emph{{On the Tate-Shafarevich group of Abelian schemes over
  higher dimensional bases over finite fields}}, manuscripta math. \textbf{150}
  (2016), 211--245 (english).

\bibitem[Kel19]{KellerGoodReduction}
Timo Keller, \emph{On an analogue of the conjecture of {B}irch and
  {S}winnerton-{D}yer for {A}belian schemes over higherdimensional bases over
  finite fields}, Doc. Math. \textbf{24} (2019), 915--993.

\bibitem[Liu06]{Liu2006}
Qing Liu, \emph{{Algebraic geometry and arithmetic curves. Transl. by Reinie
  Ern\'e.}}, {Oxford Graduate Texts in Mathematics \textbf{6}. Oxford: Oxford
  University Press. xv, 577~p.}, 2006 (english).

\bibitem[Mil80]{MilneEtaleCohomology}
James~S. Milne, \emph{{\'{E}tale cohomology.}}, {Princeton Mathematical Series.
  \textbf{33}. Princeton, New Jersey: Princeton University Press. XIII, 323
  p.}, 1980 (english).

\bibitem[Mil86]{MilneAbelianVarieties}
\bysame, \emph{{Abelian varieties.}}, {Arithmetic geometry, Pap. Conf.,
  Storrs/Conn. 1984, 103--150 (1986)}, 1986.

\bibitem[Mum70]{MumfordAbelianVarieties}
David Mumford, \emph{{Abelian varieties}}, Oxford University Press. viii, 242
  p., 1970 (english).

\bibitem[RG71]{RaynaudGruson}
Michel {Raynaud} and Laurent {Gruson}, \emph{{Crit\`eres de platitude et de
  projectivit\'e. Techniques de «\,platification\,» d'un module.}}, {Invent.
  Math.} \textbf{13} (1971), 1--89 (french).

\bibitem[{Sha}86]{ShatzGroupSchemes}
Stephen~S. {Shatz}, \emph{{Group schemes, formal groups, and p-divisible
  groups.}}, {Arithmetic geometry, Pap. Conf., Storrs/Conn. 1984, 29-78
  (1986)}, 1986.

\bibitem[Sza09]{SzamuelyGalois}
Tam\'as Szamuely, \emph{{Galois groups and fundamental groups.}}, {Cambridge
  Studies in Advanced Mathematics \textbf{117}. Cambridge: Cambridge University
  Press. ix, 270~p.}, 2009 (english).

\bibitem[use]{257441}
user19475, \emph{finitness of syntomic/fppf cohomology with coefficients in a
  finite flat group scheme}, MathOverflow, URL:
  \verb|https://mathoverflow.net/q/257441| (version: 2016-12-19).

\end{thebibliography}

\end{document}